\documentclass[final]{siamltex}

\usepackage[T1]{fontenc}
\usepackage[latin1]{inputenc}
\usepackage{amsmath}
\usepackage{amssymb}
\usepackage{amsfonts}
\usepackage{graphicx}
\usepackage{epstopdf}

\newcommand{\R}{\mathbb{R}}
\newcommand{\s}{\mathbb{S}}
\newcommand{\Z}{\mathbb{Z}}
\newcommand{\ve}{{\varepsilon}}
\newcommand{\pa}{{\partial}}

\newcommand{\EE}{{\mathcal{E}}}
\newcommand{\FF}{{\mathcal{F}}}
\newcommand{\LL}{{\mathcal{L}}}
\newcommand{\OO}{{\mathcal{O}}}
\newcommand{\MM}{{\mathcal{M}}}
\newcommand{\NN}{{\mathcal{N}}}
\newcommand{\RR}{{\mathcal{R}}}
\renewcommand{\SS}{{\mathcal{S}}}
\newcommand{\TT}{{\mathcal{T}}}
\newcommand{\UU}{{\mathcal{U}}}
\newcommand{\VV}{{\mathcal{V}}}

\DeclareMathOperator{\Diff}{D}
\DeclareMathOperator{\id}{id}

\DeclareMathOperator{\sgn}{sgn}
\DeclareMathOperator*{\range}{ran}
\DeclareMathOperator{\spn}{span}

\newtheorem{rem}[theorem]{Remark}
\newtheorem{assumption}[theorem]{Assumption}
\newtheorem{example}[theorem]{Example}

\title{Steady water waves with multiple critical layers}

\author{Mats Ehrnstr\"om\thanks{Institut f\"ur Angewandte Mathematik, Leibniz Universit\"at Hannover, Welfengarten 1, 301 67 Hannover, Germany
({\tt ehrnstrom@ifam.uni-hannover.de})} \and Joachim Escher\thanks{Institut f\"ur Angewandte Mathematik, Leibniz Universit\"at Hannover, Welfengarten 1, 301 67 Hannover, Germany
({\tt escher@ifam.uni-hannover.de})} 
\and Erik Wahl\'en\thanks{Centre for Mathematical Sciences, Lund University, PO 
Box 118, 221\,00 Lund, Sweden ({\tt erik.wahlen@math.lu.se})}}

\usepackage{color}
\definecolor{MyLinkColor}{rgb}{0,0,0.4}
\usepackage[colorlinks=true,linkcolor=MyLinkColor,urlcolor=blue,citecolor=red]{hyperref}

\begin{document}

\maketitle

\begin{abstract}
We construct small-amplitude periodic water waves with multiple critical layers. 
In addition to waves with arbitrarily many critical layers and a single crest in each period, 
multimodal waves with several crests and troughs in each period are found. 
The setting is that of steady two-dimensional finite-depth gravity water waves with vorticity.       
\end{abstract}

\section{Introduction}
The periodic steady water-wave problem describes wave-trains of two-dimensional surface water waves propagating with constant shape and velocity. 
The mathematical existence of such steady water waves is well-established, and one distinguishes between infinite and finite depth, gravitational and 
capillary waves, and flows with and without vorticity. Historically, the main focus has been on irrotational waves, i.e.\ on flows without vorticity. 
An important class of such waves is given by the \emph{Stokes waves}: periodic gravity-waves with a single crest in each period around which the wave is symmetric. 
Of particular mathematical interest is the so-called \emph{wave of greatest height} with a sharp crest of interior angle $120$ degrees. 
The paper \cite{MR1422004} describes much of that theory; see also \cite{MR2097656} for a general overview of the water-wave problem. 

A few years ago the authors of \cite{ConstantinStrauss04} proved the existence of water waves with a general vorticity distribution 
(for a different approach, cf.\ \cite{MR2500496}). 
This opened the way, not only for a new branch of research, but also for the existence of new types of waves. In many aspects, the features of rotational 
waves resemble those of irrotational waves: results on symmetry~\cite {MR2362244,MR2256915,MR2329142}, regularity~\cite{EC2010}, and waves of 
greatest height~\cite{Varvaruca20094043} have now been extended from irrotational waves to waves with a general vorticity distribution. 
In most of those cases the differences have manifested themselves as technical difficulties, not different properties \emph{per se}.      

An exclusion from this rule is the appearance of internal \emph{stagnation points}---fluid particles with zero velocity relative to the 
steady flow---and \emph{critical layers}. By a critical layer we mean a connected part of the fluid domain consisting only of closed streamlines.
The existence of such waves within linear theory was established in~\cite{MR2409513}, where the streamlines and particle paths of waves 
with constant vorticity were studied. For irrotational flows such behavior is not encountered, cf.\ \cite{MR2257390, Constantin:2010lr}. 
Also earlier studies had indicated that internal stagnation should be possible (see, e.g., \cite{MR851672}), but it was first with~\cite{Wahlen09} 
that waves with critical layers were given a firm mathematical foundation (see also \cite{CV09}). A different approach in dealing with the free boundary allowed for including 
internal stagnation, a phenomenon which for technical reasons had not been possible to see with earlier methods for rotational waves (cf., e.g., \cite{ConstantinStrauss04, Wahlen06b}).         

Another feature that separates waves with vorticity from those without is that the corresponding linear problem allows for multi-dimensional kernels. 
Whereas capillarity naturally introduces more degrees of freedom in the linear problem \cite{JonesToland86, Jones89, JonesToland85, Wahlen06b,MR2292728,W09}, 
periodic irrotational gravity-waves allow only for kernels with a single Fourier mode \cite{07022009}. Using a result from inverse spectral theory, 
it was shown in \cite{07022009} that one can construct rotational background flows for which the kernel of the linear problem contains several 
Fourier modes, even in the absence of surface tension. It was, however, not clear to what extent those flows would be ``natural'' or ``exotic'', 
or whether they would allow for exact nonlinear waves. 

In this paper we extend the results from \cite{Wahlen09} and combine them with those from \cite{07022009}. 
The novelty is two-fold: we construct i) waves with arbitrarily many critical layers, and ii) families of waves 
bifurcating from kernels with two different Fourier modes. The latter type of wave displays a more complicated surface profile than those of 
Stokes waves (cf. Figure~\ref{fig:two-dim}, p.~\pageref{fig:two-dim}). This qualitatively different behavior arises not as a result of involved 
rotational flow, but the vorticity functions are linear maps. Thus, the paper at hand is a natural extension of \cite{Wahlen09}, which dealt with 
flows of constant vorticity. We settle here for the mere existence of such waves. A qualitative investigation of the small-amplitude waves 
corresponding to the first Fourier mode will be given in a separate study \cite{eev10}.     

For a given basic wave number $\kappa>0$, our problem has three non-dimensional parameters. 
Treating the first of these as a bifurcation parameter, while fixing the other two, we prove that a curve of solutions of the nonlinear problem bifurcates from any point where the 
kernel of the linearized problem is one-dimensional using the bifurcation theory of Crandall and Rabinowitz \cite{CrandallRabinowitz71}. 

At points where the kernel is two-dimensional we similarly fix the third parameter and treat the other two as bifurcation parameters. 
Assuming that the linearized problem has two solutions with wave numbers $k_1, k_2 \in \kappa\Z_+$, such that $k_1<k_2$ and $k_2/k_1\not \in \Z$, 
we obtain in this way three two-dimensional sheets of bifurcating solutions. The first and the second of these two sheets consist of ``regular'' solutions of minimal period 
$2\pi/k_1$ and $2\pi/k_2$, respectively, with one crest and one trough per period, while the third consists of multimodal solutions, that is, 
solutions with several crests and troughs per period (and with minimal period greater than $2\pi/k_1$), except for the points where the sheets intersect. 
The points of intersection can be interpreted as secondary symmetry-breaking bifurcation points along branches of regular solutions with period 
$2\pi/k_1$ or $2\pi/k_2$. We also obtain some results in the more complicated case when $k_2/k_1\in \Z$. 
Our approach is based on the Lyapunov-Schmidt reduction and the implicit function theorem. The same method was recently used to study the bifurcation of steady irrotational water waves under a 
heavy elastic membrane \cite{BaldiToland10}.

For a comparison with the irrotational case, we would like to draw the attention to the investigation \cite{BDT00}. There a global bifurcation theory for Stokes waves was developed. 
The authors prove that the main bifurcation branch of, say, $2\pi$-periodic Stokes waves contains countably many subharmonic bifurcation points, from which waves with minimal period 
$2\pi q$, $q$ a prime number, bifurcate. These bifurcations do not take place in the vicinity of laminar flows, but as one approaches the wave of greatest height along the 
main bifurcation branch. 
In the presence of surface tension, however, the linear problem has a two-dimensional kernel for certain critical values of the surface tension parameter. 
The associated bifurcation problem has been studied in detail \cite{JonesToland86, Jones89, JonesToland85}. 
Using the Lyapunov-Schmidt reduction and blowing-up techniques the problem can be reduced to a finite-dimensional system of polynomial equations. 
For fixed values of the surface tension, branches of regular solutions bifurcate when the wave speed is used as bifurcation parameter. Multimodal solutions appear either as primary bifurcation branches or through secondary bifurcations from primary branches of regular solutions, depending on the value of the bifurcation parameter. 

The plan of this paper is as follows. Section~\ref{sec:preliminaries} describes the mathematical problem, the laminar background flows and the transformation of the free-boundary problem to a fixed domain. In Section~\ref{sec:functional} we fix the functional-analytical setting and introduce some auxiliary maps. Section~\ref{sec:bifurcation} contains the main results, which in Section~\ref{sec:applications} are applied to yield three different families of solutions to the original problem.  

\section{Preliminaries}\label{sec:preliminaries}

The {\em steady water-wave problem} is to find a surface profile $\eta$ and a velocity field $(u,v)$, defined in the fluid domain
$\Omega=\{ (x,y)\in \R^2\colon 0<y<d+\eta(x)\}$, where $d$ is the depth of the undisturbed fluid, satisfying Euler's equations 
\begin{subequations}\label{eq:waterproblem}
\begin{equation}
\label{eq:Euler}
\begin{aligned}
(u-c)u_x+vu_y &= -P_x \\
(u-c)v_x+vv_y &= -P_y-g 
\end{aligned}\qquad \text{ in } \Omega, 
\end{equation}
and 
\begin{equation}
\label{eq:massconservation}
u_x+v_y=0
\qquad \text{in } \Omega, 
\end{equation}
together with the kinematic boundary conditions
\begin{equation}
\label{eq:kinematicbottom}
v=0
\qquad \text{at }\,  y=0 
\end{equation}
and
\begin{equation}
\label{eq:kinematicsurface}
v=(u-c)\eta_x
\qquad \text{at }\, y=d+\eta(x), 
\end{equation}
as well as the dynamic boundary condition
\begin{equation}
\label{eq:dynamicsurface}
P=0 \qquad \text{at }\, y=d+\eta(x).
\end{equation}
\end{subequations}
Here $P$ is the pressure, $c>0$ is the wave speed and $g$ is the gravitational constant of acceleration. 
\medskip

\subsubsection*{Stream-function formulation}
Denote by $B$ the flat bed $y=0$ and the by $S$ the free surface $y = d + \eta(x)$. The continuity equation~\eqref{eq:massconservation} implies the existence of a {\em stream function} $\psi$ with 
the properties that $\psi_x=-v$ and $\psi_y=u-c$. 
The {\em vorticity} of the flow, $\omega:=v_x-u_y$, describes the rotation of the water particles. Note that $\omega=-\psi_{xx}-\psi_{yy}$.
The assumption that $u<c$ in $\Omega$ implies the existence of a function $\gamma$, such that
$\omega=\gamma(\psi)$ (see \cite{ConstantinStrauss04}). 
Under this assumption, the equations \eqref{eq:waterproblem} are equivalent to the following boundary value problem for $\psi$:
\begin{equation}
\label{eq:psiproblem}
\begin{aligned}
\psi_{xx}+\psi_{yy}&=-\gamma(\psi) \quad &&\text{in} \quad \Omega,\\
\psi&=m_0 \quad &&\text{on} \quad B, \\
\psi&=m_1 \quad &&\text{on} \quad S, \\
{\textstyle \frac{1}{2}} |\nabla \psi|^2+ g\eta &=Q \quad &&\text{on} \quad S,  
\end{aligned}
\end{equation}
with $\psi_y< 0$ in $\Omega$. Here $m_0$ and $m_1$ are real constants and the {\em vorticity function} $\gamma$ is an arbitrary real-valued function defined on the interval $[m_1,m_0]$. 
The difference $m_1-m_0$ is the \emph{relative mass flux} and $Q$ is a constant related to the energy. Dropping the assumption that 
$u-c<0$ one still finds that a solution of the problem \eqref{eq:psiproblem} generates a solution 
of \eqref{eq:waterproblem}. The simplest example of a non-zero vorticity 
function is a constant function. This case was first investigated in \cite{MR2409513} for linear waves, and then put on a rigorous foundation in \cite{Wahlen09}. 
Here we treat the second easiest example, namely that of an affine function $\gamma$. This choice of $\gamma$ is primarily motivated by its mathematical simplicity rather than by any physical 
arguments. 
However, an affine function can of course be seen as a Taylor approximation of a more general $\gamma$ and in this sense our choice seems natural.

\subsubsection*{Linear vorticity function}
Notice that \eqref{eq:psiproblem} with an affine, non-constant vorticity function can be reduced to the same problem with a linear vorticity function by redefining $\psi$. 
We therefore consider the following problem, where $\alpha$ is an arbitrary constant 
and we have normalized gravity to unit size:
\begin{equation}
\label{eq:problem}
\begin{aligned}
\psi_{xx}+\psi_{yy} &= \alpha \psi \qquad &&\text{in} \: &\Omega,\\
\psi &= m_0 \qquad &&\text{on} \: &B, \\
\psi &= m_1 \qquad &&\text{on} \: &S, \\
{\textstyle \frac{1}{2}} |\nabla \psi|^2+\eta &=Q \qquad &&\text{on} \: &S.  
\end{aligned}
\end{equation}
We also normalize the undisturbed depth to unity, so that the surface $S$ is given by $y = 1 + \eta(x)$ and the fluid domain 
by $\Omega=\{ (x,y)\in \R^2\colon 0<y<1+\eta(x)\}$.
We pose the problem for $\eta \in C^{2+\beta}_{\text{even}}(\s,\R)$ and $\psi \in C^{2+\beta}_{\text{per,even}}(\overline \Omega, \R)$,
where the subscripts \emph{per} and \emph{even} denotes $2\pi/\kappa$-periodicity and evenness in the horizontal variable, 
and we have identified $2\pi/\kappa$-periodic functions with functions defined on the unit circle $\s$. 
Here, $\kappa$ is a positive real number. We also require that $\min \eta > -1$. The parameter $\alpha$ influences the nature of \eqref{eq:problem} in a substantial way. 
For example, the trivial solutions are monotone in $y$ when $\alpha>0$ whereas they oscillate when $\alpha$ is sufficiently negative (cf.\ \eqref{eq:trivial}). 
In order to simplify the discussion below we therefore make the following assumption.
\begin{assumption}
\label{assumption 1}
We assume that $\alpha<0$.
\end{assumption}

\noindent When $\alpha=0$ the vorticity function is constant and one cannot transform \eqref{eq:psiproblem} 
to an equivalent problem with a linear vorticity function. This case is, however, covered in \cite{Wahlen09}. 
When $\alpha >0$ there are only simple bifurcation points (that is, with one-dimensional kernels). 
It is not difficult to prove that a curve of nontrivial, regular solutions bifurcates from each of these points, 
just like in Theorem \ref{thm:one-dimensional}. 

\subsubsection*{Laminar flows}
The solutions constructed in this paper are perturbations of laminar
flows in which the velocity field is horizontal but depth-dependent. 
The laminar flows are 
trivial solutions of problem \eqref{eq:problem} in the sense that 
$\eta=0$ and $\psi$ is independent of $x$. Mathematically, they are given 
by the formula 
\begin{equation}
\label{eq:trivial}
\psi_0(y;\mu, \lambda, \alpha) := \mu \cos(\theta_0 (y-1)+\lambda),
\end{equation}
where $\mu, \lambda \in \R$ are arbitrary constants and 
\[
\theta_k:=|\alpha+k^2|^{1/2}, \qquad k \in \R.
\]
The numbers $Q=Q(\mu, \alpha, \lambda)$, $m_0=m_0(\mu, \alpha, \lambda)$ 
and  $m_1=m_1(\mu,\lambda)$ 
are determined so that the function
$\psi$ defined by \eqref{eq:trivial} solves equation
\eqref{eq:problem}, i.e.
\begin{equation}\label{eq:Q}
Q(\mu, \lambda, \alpha) := \frac{\mu^2\theta_0^2\sin^2(\lambda)}{2}
\end{equation}
and
\[
m_1(\mu, \lambda) := \mu\cos(\lambda), \quad m_0(\mu, \alpha, \lambda) := \mu\cos(\lambda-\theta_0)
\]  
We make the following assumption on the parameters.


\begin{assumption}
\label{assumption 2}
We assume that $(\psi_0)_y(1)\ne 0$, that is $\mu\theta_0\sin(\lambda)\ne 0$.
\end{assumption}

\noindent It can be shown that without this assumption the linearized operator $\Diff_w \FF$ appearing in the bifurcation problem is not Fredholm; see Remark \ref{rem:Assumption 2}.

\subsubsection*{The flattening transform}
Problem \eqref{eq:problem} is a free-boundary problem in the sense that $\eta$ is \emph{a priori} unknown.
In order to find a secure functional-analytic setting in which to study the problem 
we use the change of variables
\[
(x,y)\mapsto \left(x,\frac{y}{1+\eta(x)}\right),
\]
which maps $\Omega$ onto the strip 
$\hat \Omega=\{(x,s)\in \R^2\colon s\in (0,1)\}$.
In the coordinates $(x,s)$ the problem~\eqref{eq:problem} takes the form
\begin{equation}
\label{eq:flatproblem}
\begin{aligned}
\left(\hat \psi_{x}-\frac{s\eta_x\hat\psi_s}{1+\eta}\right)_x-\frac{s\eta_x}{1+\eta}\left(\hat \psi_{x}-\frac{s\eta_x\hat\psi_s}{1+\eta}\right)_s+\frac{\hat\psi_{ss}}{(1+\eta)^2} &= \alpha \hat \psi  &&\text{in }   \hat \Omega,\\
\hat \psi &= m_0  &&\text{on } s=0, \\
\hat \psi &= m_1  &&\text{on }  s=1, \\
\frac12 \left(\hat\psi_x-\frac{s\eta_x\hat\psi_s}{1+\eta}\right)^2+\frac{\hat \psi_s^2}{2(1+\eta)^2}+\eta &= Q && \text{on }  s=1.  
\end{aligned}
\end{equation}
Let $\EE((\eta,\hat \psi),\alpha,Q) = 0$ denote the first and last equation of \eqref{eq:flatproblem}. 
We have the following equivalence of the two problems:

\begin{lemma}[Equivalence]
\label{lemma:flattening}
When $\min \eta(x) > -1$ the steady water-wave problem~\eqref{eq:problem} is equivalent to that $\EE((\eta,\hat \psi),\alpha,Q) = 0$ for 
\[
(\eta, \hat \psi) \in C^{2+\beta}_{\text{even}} (\s,\R) \times \left\{ C^{2+\beta}_{\text{per,even}}
(\overline{\hat \Omega}, \R)\colon \hat\psi|_{s=0} = m_0, \hat \psi|_{s=1} = m_1\right\},
\]
and  $\EE((0,\hat\psi),\alpha, Q) = 0$ for $\hat \psi = \hat \psi(s)$ if and only if $\hat \psi(s) = \psi_0(s)$. 
\end{lemma}
\begin{proof}
Since $1 + \eta(x) \neq 0$, $\eta$ is of class $C^{2+\beta}$, and H{\"o}lder spaces are Banach algebras, the identification $\hat \psi(x,s) = \hat \psi(x, y/(1+\eta(x))) = \psi(x,y)$ gives the same regularity of the functions $\psi$ and $\hat \psi$ up to order $2+\beta$. 
The equivalence then follows from direct calculation, and one can see from \eqref{eq:flatproblem} that its $x$-independent solutions are also given by the formula \eqref{eq:trivial} with $y$ replaced by $s$. 
\end{proof}

\section{Functional-analytic setting}\label{sec:functional}

\subsubsection*{The map $\FF$}
We want to linearize the problem $\EE((\eta,\hat \psi),\alpha, Q) = 0$ around a trivial solution $\psi_0$, whence we introduce a disturbance $\hat \phi$ through $\hat\psi = \psi_0 + \hat \phi$, and the function space 
\[
X := X_1 \times X_2 :=  C^{2+\beta}_{\text{even}}(\s) \times \Bigl\{ \hat \phi \in C_\text{per, even}^{2+\beta}(\overline{\hat \Omega}) \colon \hat \phi|_{s=1} = \hat \phi|_{s=0}=0 \Bigr\}.
\]
Furthermore, we define the target space 
\[
Y := Y_1\times Y_2 := C_\text{even}^{1 +\beta}(\s) \times  C_\text{per, even}^{\beta}(\overline{\hat \Omega}),
\]
and the sets
\[
\OO := \{ (\eta,\hat \phi) \in X \colon \min \eta > -1 \}
\]
and
\[
\UU := \{ (\mu, \alpha, \lambda) \in \R^3 \colon \mu \ne 0, \alpha < 0, \sin(\lambda)\ne 0 \},
\]
which conveniently captures Assumptions \ref{assumption 1} and \ref{assumption 2}.
Then $\OO \subset X$ is an open neighbourhood of the origin in $X$, and the embedding $X \hookrightarrow Y$ is compact. Elements of $Y$ will be written $w = (\eta,\hat \phi)$ and
elements of $\UU$ will be written $\Lambda=(\mu, \alpha, \lambda)$. 
Define the operator 
$\FF\colon \OO \times \UU \to Y$ by
\[
\FF(w, \Lambda) := (\FF_1(w,\Lambda), \FF_2(w,\Lambda))
\]
where  
\begin{align*}
\FF_1(w, \Lambda) &:=
\frac12\left[ 
\left(\hat \phi_x-\frac{s\eta_x((\psi_0)_s(s;\Lambda)+\hat\phi_s)}{1+\eta}\right)^2+\frac{((\psi_0)_s(s;\Lambda)+\hat \phi_s)^2}{(1+\eta)^2}\right]_{s=1}\\
&\quad +\eta-Q(\Lambda),
\end{align*}
and
\begin{align*}
\FF_2(w,\Lambda) &:=
\left(\hat \phi_{x}-\frac{s\eta_x((\psi_0)_s(s;\Lambda)+\hat\phi_s)}{1+\eta}\right)_x\\
&\quad -\frac{s\eta_x}{1+\eta}\left(\hat \phi_{x}-\frac{s\eta_x((\psi_0)_s(s;\Lambda)+\hat\phi_s)}{1+\eta}\right)_s\\
&\quad +\frac{(\psi_0)_{ss}(s;\Lambda)+\hat\phi_{ss}}{(1+\eta)^2} - \alpha \left( \psi_0(s; \Lambda) + \hat \phi \right).
\end{align*}
We write $C^\omega$ to denote analyticity between Banach spaces. 

\begin{lemma}[Equivalence]
\label{lemma:equivalence}
The problem $\FF((\eta,\hat \phi),\Lambda)= 0$, $(\eta,\hat\phi) \in \OO$,  is equivalent to the water-wave problem \eqref{eq:problem}. 
The map $(\eta,\hat\phi) \mapsto (\eta,\hat\psi)$ is continuously differentiable, $\FF((0,\hat \phi),\Lambda) = 0$ if and only if $\hat \phi = 0$, and $\FF \in C^{\omega}(\OO \times \UU, Y)$.
\end{lemma}
\begin{proof}
Relying on Lemma~\ref{lemma:flattening} it suffices to note that $\FF((\eta,\hat \phi),\Lambda) = \EE((\eta, \psi_0(s;\Lambda)+\hat \phi),\alpha, Q(\Lambda))$. 
The analyticity follows from the fact that $\psi_0 \in C^\omega(\R \times \UU,\R)$, and that $\FF$ depends polynomially on $\hat \phi$ and $\psi_0$, 
whereas it is a rational function in $1+ \eta$.
\end{proof}

\subsubsection*{Linearization}
We obtain the linearized water-wave problem by taking the Fr\'echet derivative of $\FF$ at the point $w=0$:
\begin{align}
\label{eq:DF1}
\Diff_w \FF_{1}(0,\Lambda)w&=\left[(\psi_0)_s\, \hat\phi_s - (\psi_0)_s^2\, \eta + \eta\right]\Big|_{s=1},
\intertext{and} 
\label{eq:DF2}
\Diff_w \FF_{2}(0,\Lambda)w &= \hat \phi_{xx}+\hat \phi_{ss}-\alpha \hat \phi-s(\psi_0)_s\eta_{xx}-2(\psi_0)_{ss}\eta,
\end{align}
where $\Diff_w$ denotes the Frech{\'e}t derivative with respect to the first variable $w$.
It turns out that there is a simple transformation that transforms this linearization to the one which is obtained by formally linearizing directly in \eqref{eq:problem}. To that aim, let
\[
\tilde X_2 := \Bigl\{ \phi \in C_\text{per, even}^{2+\beta}(\overline{\hat \Omega}) \colon  \phi|_{s=0}=0 \Bigr\},
\]
and
\[
\tilde X := \left\{ (\eta,\hat \phi) \in X_1 \times \tilde X_2 \right\}.
\]
Then ${X_2 \hookrightarrow \tilde X_2}$ and $X \hookrightarrow \tilde X \hookrightarrow  Y$, where the last embedding is compact. 

\begin{proposition}[The $\TT$-isomorphism]
\label{prop:T}
The bounded, linear operator $ \TT(\Lambda) \colon \tilde X_2\to X$ given by
\[
 \TT(\Lambda) \phi := \left(-\frac{\phi|_{s=1}}{(\psi_0)_s(1)},\phi - \frac{s(\psi_0)_s\, \phi|_{s=1}}{(\psi_0)_s(1)}\right)
\]
is an isomorphism. Define $\LL(\Lambda) := \Diff_w \FF(0,\Lambda)  \TT(\Lambda)\colon \tilde X_2 \to Y$. Then
\begin{equation}
\label{eq:Lexpression}
\LL(\Lambda)\phi=\left(\, \left[(\psi_0)_s \phi_s- \left((\psi_0)_{ss}+\frac{1}{(\psi_0)_s}\right)\phi\right]\biggl|_{s=1}, \,\,  (\pa_x^2+\pa_s^2 -\alpha)\phi\, \right).
\end{equation}
\end{proposition}

{\em Proof}. 
Note that we have the relation 
\[
(\pa_x^2+\pa_s^2-\alpha)(s(\psi_0)_s \, \eta) = s(\psi_0)_s \, \eta_{xx}+2(\psi_0)_{ss} \, \eta.
\]
Thus, the right hand side of equation \eqref{eq:DF2}
can be written
\[
(\pa_x^2+\pa_s^2 -\alpha)(\hat\phi-s(\psi_0)_s\, \eta) = (\pa_x^2+\pa_s^2 -\alpha)\phi,
\]
where $\phi=\hat\phi-s(\psi_0)_s\, \eta$.
Moreover, for $\hat\phi \in X_2$ we have that
\[
\begin{aligned}
\Diff_w \FF_{1}(0,\Lambda)w &= [(\psi_0)_s\, \phi_s + (\psi_0)_s(\psi_0)_{ss}\,\eta + \eta]|_{s=1},\\
\phi|_{s=1} &= -(\psi_0)_s(1)\, \eta,\\
\phi|_{s=0} &= 0.
\end{aligned}
\]
Due to Assumption~\ref{assumption 2}  we can thus express $\eta$ in terms of $\phi$ and obtain that
\[
\Diff_w \FF_{1}(0,\Lambda)w = \left[(\psi_0)_s\, \phi_s - \left((\psi_0)_{ss} + \frac{1}{(\psi_0)_s}\right)\phi\right]\biggl|_{s=1}.
\eqno \endproof
\]

\begin{rem}
\label{rem:notation}
When the dependence on the parameters is unimportant, we shall for convenience refer to $\Diff_w \FF(0,\Lambda)$, $\LL(\Lambda)$ and $\TT(\Lambda)$ simply as $\Diff_w \FF(0)$, $\LL$ and $\TT$.
\end{rem}

Via $\TT$, elements $\phi \in \tilde X_2$ can be ``lifted'' to elements $(\eta, \hat \phi) \in \tilde X$ using the correspondence induced by the first component of $\TT\phi$.
The following result is immediate.

\begin{proposition}[Surface projection]
\label{prop:P}
The mapping $\eta_{(\cdot)}$ defined by
\[
\tilde X_2 \ni \phi \: \stackrel{\eta_{(\cdot)}}{\mapsto} \: \eta_\phi = -\frac{\phi|_{s=1}}{(\psi_0)_s(1)} \in X_1,
\]
is linear and bounded, whence the linear map
\[
\left( \eta_{(\cdot)}, \id \right) \colon \tilde X_2 \ni \phi \: \mapsto \: (\eta_\phi,\phi) \in \tilde X
\]
is bounded.
\end{proposition}

\begin{figure}
\begin{center}
\includegraphics[width=0.4\linewidth]{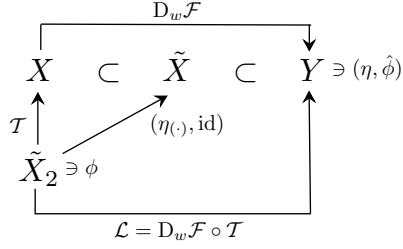}
\end{center}
\caption{\small The functional-analytic setting.}
\label{fig:diagram}
\end{figure}

We now have the functional-analytic setting outlined in Figure~\ref{fig:diagram}.
Any property of the operator $\Diff_w \FF(0)$ can be conveniently studied using the operator $\LL$. Some properties of the operator $\LL$ are recorded in the following lemma. The proof is standard and is therefore omitted (see e.g.\ \cite[Section III.1]{Kielhofer04} and \cite[Lemma 3.5 and Lemma 3.6]{Wahlen06b}).

\begin{lemma}[Fredholm properties]
\label{lemma:Lproperties}
The operator $\LL\colon \tilde X_2\to Y$ is a Fredholm operator of index $0$. Its kernel, 
$\ker \LL$,  
is spanned by a finite number of functions of the form 
\begin{equation}
\label{eq:phi_k}
\phi_k(x,s)=\begin{cases}
\cos(kx) \sin^*(\theta_k s)/\theta_k, & \theta_k\ne 0,\\
\cos(kx) s, & \theta_k = 0,
\end{cases}
\qquad k\in \kappa\Z,
\end{equation}
 where we have used the notation
\[
\sin^*(\theta_k s) := \begin{cases}\sin(\theta_k s), & k^2+\alpha<0,\\
\sinh(\theta_k s), & k^2+\alpha > 0.
\end{cases}
\]
Define $Z := \{(\eta_\phi,\phi)\colon \phi \in \ker \LL\} \subset \tilde X \subset Y$. Then the range of $\LL$, $\range \LL$, is the orthogonal complement of $Z$ in 
 $Y$ with respect to the inner product
\[
\langle w_1,w_2 \rangle_Y :=\iint_{\hat \Omega} \hat\phi_1 \hat\phi_2\, dx\, ds+\int_{-\pi}^{\pi} \eta_1\eta_2\, dx, \qquad 
w_1,w_2\in Y.
\]
The projection $\Pi_Z$ onto $Z$ along $\range \LL$ is given by 
\[
\Pi_Z w= \sum \frac{\langle w, \tilde w_k\rangle_Y}{\|\tilde w_k\|_Y^2} \tilde w_k,
\]
where the sum ranges over all $\tilde w_k=(\eta_{\phi_k},\phi_k)\in Z$, with $\phi_k$ of the form 
\eqref{eq:phi_k}.
\end{lemma}

The abbreviation $\sin^*$ in Lemma~\ref{lemma:Lproperties} will be used analogously for other trigonometric and hyperbolic functions. Noting that $\range \Diff_w \FF(0)=\range \LL$ and that
$\ker \Diff_w \FF(0)= \TT\ker \LL$ we obtain the following corollary 
of Lemma \ref{lemma:Lproperties}.

\begin{corollary}
\label{corollary:Fredholm}
$\Diff_w \FF(0)\colon X \to Y$ is a Fredholm operator of index $0$. 
\end{corollary}

\begin{rem}
\label{rem:Assumption 2}
If Assumption~\ref{assumption 2} is violated, that is if $(\psi_0)_s(1)=0$, then 
$\Diff_w \FF_1(0)w=\eta$. Consequently, the range of $\Diff_w \FF_1(0)$
equals $X_1$, and this can be used to show that $\Diff_w \FF(0)$ is not Fredholm. 
\end{rem}

\section{Bifurcation}\label{sec:bifurcation}
This section contains the main proofs of the one- and two-dimensional bifurcation.

\begin{lemma}[Bifurcation condition]
\label{lemma:Bifurcation condition}
Let $\Lambda=(\mu, \alpha,\lambda)\in \UU$, so that that Assumptions~\ref{assumption 1}--\ref{assumption 2} hold.
For $k\in \kappa\Z$ we have that $\cos(kx) \sin^*(\theta_k s)/\theta_k\in \ker \LL$ if and only if 
\begin{equation}
\label{eq:bifurcation2}
\theta_k \cot^*(\theta_k) = \frac{1}{\mu^2 \theta_0^2 \sin^2(\lambda)} + \theta_0 \cot(\lambda).
\end{equation}
\end{lemma}

\begin{rem}
When $\theta_k=0$ we interpret $\sin^*(\theta_k s)/\theta_k$ as $s$ and $\theta_k \cot^*(\theta_k)$ as $1$. We shall use this convention from now on.
\end{rem}

\begin{proof}
Assume first that $\theta_k\ne 0$.
By substituting $\phi= \cos(kx) \sin^*(\theta_k s)/\theta_k$ into \eqref{eq:Lexpression}, we find that $\LL\phi=0$ if and only if
\begin{equation}
\label{eq:bifurcation1}
\theta_0 \mu\sin(\lambda)\cos^*(\theta_k)=\left(\theta_0^2\mu \cos(\lambda)+\frac1{\mu\theta_0\sin(\lambda)}\right)\frac{\sin^*(\theta_k)}{\theta_k}.
\end{equation}
We claim that under Assumption \ref{assumption 2} this equality implies that $\sin^*(\theta_k)\ne 0$. 
Indeed, if $\sin^*(\theta_k)=0$, then $\cos^*(\theta_k)\ne 0$. Examining the left hand side of equation \eqref{eq:bifurcation1} we find  that $\theta_0\mu\sin(\lambda)=0$. But this contradicts Assumption \ref{assumption 2}. Dividing by $\theta_0 \mu\sin(\lambda) \sin^*(\theta_k)/\theta_k$ in \eqref{eq:bifurcation1} we obtain \eqref{eq:bifurcation2}.
When $\theta_k=0$ we substitute $\phi=\cos(kx)s$ into \eqref{eq:Lexpression} and similarly obtain 
\eqref{eq:bifurcation2} with left hand side $1$.
\end{proof}

Note that since we have restricted ourselves to even functions, it suffices to take $k\ge 0$. In order to find nontrivial waves we assume that $k>0$.

The bifurcation equation \eqref{eq:bifurcation2} is a transcendental equation involving the three parameters $\mu$, $\alpha$, $\lambda$ and the variable $k$. 
For fixed values of $\mu$, $\alpha$ and $\lambda$, the number of different solutions $k\in \kappa\Z_+$ gives the dimension of $\ker \FF_w(0, \Lambda)$. 
A complete description of the set of solutions and its dependence on the parameters is at the moment out of our reach. 
The next lemma shows, however, that one can choose the parameters such that there is precisely one solution $k\in \kappa\Z_+$.
We also show that the parameters can be chosen so that there are precisely two solutions $k_1, k_2\in \kappa\Z_+$. 
It is an open question if there are values of the parameters for which there are three or more solutions, that is, for which the kernel of $\FF_w(0, \Lambda)$ is at least three-dimensional. 
Note that this question only depends on the left hand side of \eqref{eq:bifurcation2}. 
For example, if for some $a\in \R$ and $\alpha<0$ there are exactly three solutions to the equation $\theta_k\cot^*(\theta_k)=a$, we can always adjust $\lambda$ and $\mu$ to make the right 
hand side equal to $a$ (this is clear from the proof of Lemma~\ref{lemma:kernels}) and thus obtain an example of a three-dimensional kernel. The fact that the function $k\mapsto \theta_k\cot^*(\theta_k)$ is 
not monotone when $\alpha <0$ is mainly what makes the analysis of the bifurcation equation difficult. On the other hand, it is exactly this property that allows for multi-dimensional kernels.

\begin{lemma}[Bifurcation kernels]
\label{lemma:kernels}
Let $k_1, k_2 \in \kappa\Z_+$ and $\Lambda \in \UU$.  
\begin{itemize}
\item[i)]  For $a.e.$ $\alpha$ and any $k_1$  there are $\mu$ and $\lambda$ such that \eqref{eq:bifurcation2} holds only for $k = k_1$.
\item[ii)] For any $\lambda$ with $\cot(\lambda)\le 0$ and $k_1,k_2$ such that $k_2^2 \ge k_1^2 + \frac{9}{4}\pi^2$, there are $\alpha \in (-k_2^2,-k_1^2 - \pi^2)$  and $\mu$ such that \eqref{eq:bifurcation2} holds for $k =k_1, k_2$, and for no other $k > k_1$.
\item[iii)] For any $k_1,k_2$ such that $k_2^2 > k_1^2 + 3\pi^2$ and $k_2^2-k_1^2\ne (2n+1)\pi^2$, for all $n\in \Z$, there are $\alpha < - k_2^2$, $\mu$ and $\lambda$ such that \eqref{eq:bifurcation2} holds at least for $k = k_1, k_2$.
\end{itemize}
The points $\alpha$ such that there exist $\mu$ and $\lambda$ for which \eqref{eq:bifurcation2} holds for more than one $k \in \kappa\Z_+$ are isolated and thus a Lebesgue null set. 
\end{lemma}

\begin{rem}\label{rem:bifurcation}
For $\alpha$ as given by~ii) and~iii) in Lemma~\ref{lemma:kernels}, the background laminar flow~\eqref{eq:trivial} and its derivative 
$(\psi_0)_y(y;\Lambda)$ has at least one zero (stagnation occurs). For $\alpha$ as in Lemma~\ref{lemma:kernels}~i), 
the presence of a stagnation point for some $y \in (0,1)$ is guaranteed if $\alpha < -\pi^2$. 

Note that in the cases ii) and iii) we generally only assert that the kernel is {\em at least} two-dimensional. However, if we choose $k_1=\kappa$ in ii) we obtain that the kernel is 
{\em precisely} two-dimensional. In Example \ref{ex:sin-sinh} it is shown that one can also obtain a two-dimensional kernel with $k_1>\kappa$.
\end{rem}

\begin{proof}
The right hand side of \eqref{eq:bifurcation2}, call it $r(\mu)$, satisfies 
\[
\range_{\mu \ne 0} (r(\mu)) = (\theta_0 \cot(\lambda),\infty) \supseteq \R_+,
\]
if we choose $\lambda$ such that $\cot(\lambda)\le 0$.
We set $t := |\alpha|$, and study the function
\begin{equation*}
h(t;k) :=
\begin{cases}
\begin{aligned}
h_1(t;k) &= \sqrt{t - k^2} \cot\left( \sqrt{t - k^2} \right), \qquad &t > k^2 > 0,\\
h_2(t;k) &= \sqrt{k^2 - t} \coth\left( \sqrt{k^2 - t} \right), \qquad &k^2 \geq t > 0.
\end{aligned}
\end{cases}
\end{equation*}
The functions $h_1$ and $h_2$ both tend to $1$ as $t-k^2$ vanishes, whence the function $h$ is continuous away from the singularities $t-k^2 = n^2 \pi^2$, $n \in \Z_+$.  

Now, let $x := \sqrt{|t-k^2|}$. Since
\begin{align*}
\frac{d}{dx}\, ( x \coth(x) ) &=  \frac{\cosh(x) \sinh(x) - x}{\sinh^2(x)} > 0, \qquad &0 <\; &x \in \R,\\
\frac{d}{dx}\, ( x \cot(x) )  &=  \frac{\cos(x) \sin(x) - x}{\sin^2(x)} < 0, \qquad  &0 <\; &x \not\in \pi\Z,
\end{align*}
it follows that $h$ is strictly increasing in $k$, as well as strictly decreasing in $t$ (which, of course, does not apply across the singularities). 
Moreover, if we let $\bar n$ be the maximal integer $n\ge 0$ such that $n^2\pi^2 < t$, then
\[
\range_{t-k^2 < t} (h) = \begin{cases}\left(\sqrt{t} \cot \sqrt{t}, \infty \right), & t< \pi^2,\\ \R, & t=\pi^2,\end{cases}
\]
for $\bar n=0$, and 
\begin{align*}
\range_{{\bar n}^2 \pi^2 < t-k^2 < t} (h) &= \begin{cases}\left(\sqrt{t} \cot \sqrt{t}, \infty \right), & t< (\bar n +1)^2 \pi^2,\\ \R, & t=(\bar n +1)^2 \pi^2,\end{cases}\\
\range_{n^2 \pi^2 < t-k^2 < (n+1)^2 \pi^2} (h) &= \R, \qquad n = 1,2, \ldots, \bar n -1,\\
\range_{t-k^2 < \pi^2} (h) &= \R
\end{align*}
otherwise.

To prove ii), consider the set $I$ of $|\alpha| = t$ such that
\begin{equation}\label{eq:t-condition}
\pi^2 < t-k_1^2 < \left(\frac{3 \pi}{2}\right)^2 \quad\text{ and }\quad \range_{t\in I} (h(t;k_1)) = (1,\infty).
\end{equation}
Due to the monotonicity with respect to $k$, the only possible $k > k_1$ such that $h(t;k_1) = h(t;k)$ must belong to the interval $k^2 > t$, i.e. to the hyperbolic regime, and there exists at most one such $k$ for a given $t$. 
In particular, if $t$ fulfills \eqref{eq:t-condition} and $k_2^2 \ge k_1^2 + \frac{9}{4}\pi^2$, then $k_2^2 - t > 0$. Since, in that case,
\begin{equation}\label{eq:aneq1}
1 < h(t;k_2) < k_2 \coth(k_2),
\end{equation}
while $h(t;k_1)$ spans $(1,\infty)$ as $t$ spans $I$, we may choose $|\alpha| = t$ such that $h_1(t;k_1) = h_2(t;k_2)$. We then choose $\mu$ such that $h(t;k_1) = r(\mu)$.

We now prove that the points $\alpha$ for which \eqref{eq:bifurcation2} holds simultaneously for at least two different $k$'s are isolated. First say that 
\begin{equation}\label{eq:f}
f(t) := h_1(t;k_1) - h_1(t;k_2), \qquad k_1 < k_2,
\end{equation}
satisfies $f(t_0) = 0$, for $t_0 > k_j^2$, $j = 1,2$. Differentiation with respect to $t$ yields that
\begin{align*}
f^\prime(t) &=  \frac {\cot \left( \sqrt {t-k_1^2} \right) }{2 \sqrt {t-{k
_{{1}}}^{2}}}  - {\frac {\cot \left( \sqrt {t-k_2^2} \right) }{2 \sqrt {t-{k_{{2}}}^{2}}}} \\
&\quad + \frac{1}{2} \cot^2 \left( \sqrt {t-{k_{{2}}}^{2}} \right) - \frac{1}{2} \cot^2 \left( \sqrt {t-{k_{{1}}}^{2}} \right),
\end{align*}
which, using the fact that $f(t_0) = 0$, at $t = t_0$ can be rewritten as
\[
f^\prime(t_0) = \frac{ \left(k_2^2 - k_1^2 \right) h_1(t_0;k_1) \left( h_1(t_0;k_1) - 1\right)}{2 \left(t_0-k_2^2 \right) \left( t_0-k_1^2 \right)}. 
\]
Since $h_1(\cdot; k)$ is locally strictly decreasing, there exists $\varepsilon > 0$ with
\[
f(t) \neq 0, \qquad 0 < |t - t_0| < \varepsilon. 
\]
A similar, although not identical, calculation for $h_1(t;k_1)$ and $h_2(t;k_2)$ yields the same conclusion when $k_1^2 < t_0 < k_2^2$. When $k_1^2 < k_2^2=t_0$ a direct computation 
shows that $f'(t_0)=-1/6$, so that the conclusion remains true (here we have extended the definition of $f(t)$ in the natural way). Since, in view of the monotonicity of $h$ for $k^2 - t > 0$, there are finitely many $k_j$'s for which one could have equality with $k_1$ in \eqref{eq:f}. Therefore
\[
\inf_{k_j \in \kappa\Z_+} \left\{ \left| h(t_0;k_j) - h(t_0;k_1) \right| \colon h(t_0;k_j) \neq h(t_0;k_1) \right\} \geq \delta > 0,
\]
and a possibly smaller $\varepsilon$ exists such that $h(t;k_i) \neq h(t;k_j)$ for $|t - t_0| \in (0, \varepsilon)$,  $i \neq j$. Hence, the points $t$ for which $f(t) = 0$, for some $k_i \neq k_j$, are isolated.

This argument also proves i): For $a.e.$ $t = |\alpha|$, we have that $h(t;k_1) \neq h(t;k_j)$, $k_j \neq k_1$.We then pick $\lambda$ such that $\theta_0 \cot(\lambda) < h(t;k_1)$ and the appropriate $\mu$ such that $h(t;k_1) = r(\mu)$.

iii) Since we are interested in $\alpha < -k_2^2$, we study $t > k_2^2$. 
By the assumptions there exists $n\in \Z_+$ such that
\[
(2n+1)\pi^2<k_2^2-k_1^2<(2n+3)\pi^2.
\]
This implies that 
\[
k_1^2+(n+1)^2\pi^2<k_2^2+n^2\pi^2<k_2^2+(n+1)^2\pi^2<k_1^2+(n+2)^2\pi^2.
\]
On the interval $(k_2^2+n^2\pi^2,k_2^2+(n+1)^2\pi^2)$ the function $h(t;k_2)$ spans $\R$, while $h(t;k_1)$ is bounded.
It follows that $h(t_0;k_1)=h(t_0;k_2)$ for some $t_0\in (k_2^2+n^2\pi^2,k_2^2+(n+1)^2\pi^2)$.
 Hence, we may find (at least) one $t = |\alpha| > k_2^2$ such that $h(t;k_1) = h(t;k_2) \in \R$, and then choose $\lambda$ and $\mu$ appropriately, so that \eqref{eq:bifurcation2} holds for $k=k_1,k_2$. 
\end{proof}

\subsubsection*{Lyapunov-Schmidt reduction}
Let $\Lambda^*$ denote a triple $(\mu^*,\alpha^*, \lambda^*)$ such that \eqref{eq:bifurcation2} holds and suppose  that the kernel is nontrivial, so that
 \[\ker \LL(\Lambda^*)=\spn\{\phi^*_1, \dots, \phi^*_n\},\] with 
 $\phi^*_j=\cos(k_j s)\sin^*(\theta_{k_j}s)/\theta_{k_j}$ and 
 $0<k_1<\cdots< k_n$. Let $w^*_j=\TT(\Lambda^*) \phi^*_j$.
>From Lemma \ref{lemma:Lproperties} it follows that $Y=Z\oplus \range \LL(\Lambda^*)$. 
As in that lemma, we let $\Pi_Z$ be the corresponding projection onto $Z$ parallel to $\range \LL(\Lambda^*)$. 
This decomposition induces similar decompositions $\tilde X=Z \oplus (\range \LL(\Lambda^*) \cap \tilde X)$ and 
$X=\ker \FF_w(0,\Lambda^*) \oplus X_0$, where $X_0=\RR(\range \LL(\Lambda^*) \cap \tilde X)$ in which 
\[
\RR(\eta, \phi)=\left(\eta, \phi-\frac{s(\psi_0)_s\phi|_{s=1}}{(\psi_0)_s(1)}\right).
\]
Applying the Lyapunov-Schmidt reduction \cite[Thm I.2.3]{Kielhofer04} we obtain the following lemma.

\begin{lemma}
\label{lemma:Lyapunov-Schmidt}
There exist open neighborhoods $\NN$ of $0$ in $\ker \FF_w(0,\Lambda^*)$, $\MM$ of $0$ in $X_0$ and $\UU' $ of $\Lambda^*$ in $\R^3$, and a function
 $\psi \in C^\infty(\NN \times \UU' , \MM)$, 
such that
\[
\FF(w,\Lambda)=0 \quad\text{ for }\quad w\in \NN+\MM,\quad \Lambda\in \UU',
\]
if and only if $w=w^*+\psi(w^*, \Lambda)$ and $w^*=t_1 w^*_1+\cdots+t_n w_n^*\in \NN$ solves the 
finite-dimensional problem
\begin{equation}
\label{eq:Lyapunov-Schmidt}
\Phi(t, \Lambda)=0 \quad\text{ for }\quad t \in \VV, \quad \Lambda \in \UU',
\end{equation}
in which
\[
\Phi(t, \Lambda):=\Pi_Z \FF(w^*+\psi(w^*, \Lambda), \Lambda), \quad 
\]
and $\VV := \{t \in \R^n \colon t_1 w^*_1+\cdots+t_n w_n^*\in \NN\}$. The function $\psi$ has the properties $\psi(0,\Lambda)=0$ and $D_w \psi(0, \Lambda)=0$.
\end{lemma}

\subsubsection*{Bifurcation from a one-dimensional kernel}

\begin{theorem}[One-dimensional bifurcation]
\label{thm:one-dimensional}
Suppose that
\begin{equation*}
\dim \ker \Diff_w \FF(0,\Lambda^*)=1,
\end{equation*}
and let $\ker \Diff_w \FF(0,\Lambda^*)=\spn \{w^*\}$.
There exists a $C^\infty$-curve of small-amplitude nontrivial solutions $\{(\overline w(t),\overline \mu(t))\colon 0<|t|<\ve\}$ 
of 
\begin{equation}
\label{eq:problem one-dimensional}
\FF(w,\mu,\alpha^*,\lambda^*)=0
\end{equation}
in $\OO \times \R$, passing 
through $(\overline w(0), \overline \mu(0))=(0, \mu^*)$ with 
\[
\overline w(t)=tw^*+O(t^2)
\] 
in $\OO$ as $t \to 0$. In a neighborhood of $(0, \mu^*)$ in $\OO \times \R$ these are the only 
nontrivial solutions of \eqref{eq:problem one-dimensional}.
The corresponding surface profiles have one crest and one trough per period, and are strictly monotone between crest and trough. 
\end{theorem}

\begin{proof}
We apply the local bifurcation theorem with a one-dimensional kernel by Crandall and Rabinowitz \cite{CrandallRabinowitz71}. 
We shall repeat the details here in order to clarify the proof in the case of a two-dimensional kernel below.
By Lemma \ref{lemma:Lyapunov-Schmidt} the equation $\FF(w,\Lambda)=0$ is locally equivalent to 
the finite-dimensional problem $\Phi(t, \Lambda)=0$ where $t\in \R$. Note that $\Phi(t,\Lambda)=\Phi_1(t,\Lambda)\tilde w^*$ where $\tilde w^*=(\eta_{\phi^*}, \phi^*)$, $\phi^*=\TT^{-1}(\Lambda^*)w^*$, and $\Phi_1$ is a real-valued function.
Since $\Phi_1(0,\Lambda)=0$ we find that $\Phi_1(t,\Lambda)=t\Psi(t,\Lambda)$ in which
\[
\Psi(t,\Lambda)=\int_0^1 (\pa_t \Phi_1)(zt, \Lambda)\, dz.
\]
For $t\ne 0$ the equations $\Phi_1=0$ and $\Psi=0$ are equivalent.
We now prove that $\pa_\mu \Psi(0,\Lambda^*)\ne 0$ and apply the implicit function theorem to $\Psi$.
We have that
\[
\pa_\mu \Psi(0, \Lambda^*)=\pa_\mu \pa_t \Phi_1(0,\Lambda^*).
\]
Moreover, 
\begin{align*}
\pa_\mu\pa_t \Phi(0,\Lambda^*)&=\Pi_Z \Diff^2_{w\mu} \FF(0, \Lambda^*) w^*\\
&=\frac{\langle \Diff^2_{w\mu} \FF(0,\Lambda^*) w^*, \tilde w^*\rangle_Y}{\| \tilde w^*\|_Y^2} \tilde w^*.
\end{align*}
Hence, $\pa_\mu \Psi(0, \Lambda^*)\ne 0$ if and only if $\langle \Diff^2_{w\mu} \FF(0,\Lambda^*) w^*, \tilde w^*\rangle_Y\ne 0$.
Recall that
$\Diff_w \FF(0,\Lambda)=\LL(\Lambda)\TT^{-1}(\Lambda)$. Thus, 
\begin{equation}
\label{eq:derivativesum}
\Diff^2_{w \mu} \FF(0,\Lambda^*)w^*=\Diff_\mu \mathcal L(\Lambda^*)\phi^*+\LL(\Lambda^*)\Diff_\mu \TT^{-1}(\Lambda^*) w^*.
\end{equation}
Since $\LL = \Diff_w \FF(0) \circ \TT$ the second term in the right hand side of \eqref{eq:derivativesum} belongs to $\range \Diff_w \FF(0,\Lambda^*)$, and we find that
$\langle \Diff^2_{w\mu} \FF(0,\Lambda^*) w^*, \tilde w^*\rangle_Y=
\langle \Diff_\mu \LL(\Lambda^*)\phi^*, \tilde w^*\rangle_Y$. 
A straightforward calculation shows that
\begin{align*}
&\Diff_\mu\mathcal L(\Lambda^*)\phi^*\\
&=\left(\left[-\theta_0\sin(\lambda^*)\cos^*(\theta_k)-\left(-\theta_0^2\cos(\lambda^*)+\frac{1}{(\mu^*)^2\theta_0 \sin(\lambda^*)}\right)\frac{\sin^*(\theta_k)}{\theta_k}\right]\cos(k x),0\right),
\end{align*}
so that
\begin{equation}
\label{eq:partialLmu}
\begin{aligned}
&\langle \Diff_\mu \LL(\Lambda^*)\phi^*, \tilde w^*\rangle_Y\\
&=
\frac{\pi}{(\psi_0)_s(1)} \left[\theta_0\sin(\lambda^*)\cos^*(\theta_k)+\left(-\theta_0^2\cos(\lambda^*)+\frac{1}{(\mu^*)^2\theta_0 \sin(\lambda^*)}\right)\frac{\sin^*(\theta_k)}{\theta_k}\right]\frac{\sin^*(\theta_k)}{\theta_k}.
\end{aligned}
\end{equation}
Using the fact that $\sin^*(\theta_k)/\theta_k \ne 0$ (see the proof of Lemma \ref{lemma:Bifurcation condition}) and 
rearranging, we obtain that the right hand side in \eqref{eq:partialLmu} vanishes if and only if 
\[
\theta_k \cot^*(\theta_k) =- \frac{1}{(\mu^*)^2 \theta_0^2 \sin^2(\lambda^*)} + \theta_0 \cot(\lambda^*),
\]
which contradicts \eqref{eq:bifurcation2}. 
\end{proof}

\begin{rem}\label{rem:two-dim}
It follows from the implicit function theorem and the bifurcation equation~\eqref{eq:bifurcation2} that as $\alpha$ and $\lambda$ are varied in a neighborhood of $(\alpha^*, \lambda^*)$ one obtains a whole two-dimensional 
family of bifurcating curves which depend smoothly on $\alpha$ and $\lambda$.
\end{rem}

\subsubsection*{Bifurcation from a two-dimensional kernel}

When the kernel is two-dimensional we can still apply Theorem~\ref{thm:one-dimensional} by restricting our attention to functions that are 
$2\pi/k_2$-periodic in the $x$-variable. Let $X^{(k_2)}$ denote the subspace (subset) of such functions in any (open set in a) Banach space $X$. 
In particular, the restriction $\FF^{(k_2)} := \FF|_{\OO^{(k_2)}}$ is well-defined, and, in view of that $k_2 > k_1$, we have that 
$\ker \Diff_w \FF^{(k_2)}(0,\Lambda^*) = \spn \{w_2^* \} := \TT(\Lambda^*) \spn \{\phi_2^*\}$.  
An application of Theorem~\ref{thm:one-dimensional} and Remark~\ref{rem:two-dim} to $\FF^{(k_2)}$ yields that the set of nontrivial solutions 
of 
\begin{equation}
\label{eq:problem two-dimensional}
\FF(w,\mu, \alpha,\lambda^*)=0
\end{equation}
in $\OO^{(k_2)}\times \R^2$ is locally given by the two-dimensional sheet \[\SS^{(k_2)}=\{(\overline w_2(t,\alpha),\overline \mu_2(t,\alpha),\alpha)\colon (t,\alpha)\in \VV_2, t\ne 0\},\] where $\VV_2$ is an open neighborhood of 
$(0, \alpha^*)$ in $\R^2$. When $k_2$ is not an integer multiple of $k_1$ the same argument gives a set $\SS^{(k_1)}$ of nontrivial solutions of \eqref{eq:problem two-dimensional} in $\OO^{(k_1)}\times \R^2$.
It is however also possible to obtain solutions that are neither in $X^{(k_1)}$ nor in $X^{(k_2)}$, as we will now describe.

\begin{theorem}[Two-dimensional bifurcation]\label{thm:two-dimensional}
Suppose that
\[
\dim \ker \Diff_w \FF(0,\Lambda^*)=2.
\]
Define
\begin{equation}\label{eq:a}
a:= \theta_{k_1}\cot^*(\theta_{k_1})= \theta_{k_2}\cot^*(\theta_{k_2})
\end{equation}
as the left-hand side of~\eqref{eq:bifurcation2}, and assume that that $a \not \in \{0,1\}$, or that $\theta_{k_2}=0$ (in which case $a=1$). 
\medskip

\begin{itemize}
\item[i)]
If $k_2/k_1\not \in \Z$ there exists a smooth sheet of small-amplitude nontrivial solutions 
\[
\SS^{\text{mixed}} := \{(\overline w(t_1,t_2),\overline \mu(t_1,t_2), \overline \alpha(t_1,t_2))\colon 0<t_1^2+t_2^2<\ve^2\}
\] 
of \eqref{eq:problem two-dimensional} in $\OO \times \R^2$, passing 
through $(\overline w(0,0), \overline \mu(0,0), \overline \alpha(0,0))=(0, \mu^*, \alpha^*)$ with 
\[
\overline w(t_1,t_2)=t_1 w_1^*+t_2 w_2^* +O(t_1^2+t_2^2)
\]
In a neighborhood of $(0, \mu^*, \alpha^*)$ in $\OO \times \R^2$ the union $\SS^\text{mixed} \cup \SS^{(k_1)}\cup \SS^{(k_2)}$  contains all nontrivial solutions of \eqref{eq:problem two-dimensional}. 
\medskip

\item[ii)]
Let $\delta > 0$. If $k_2/k_1\in \Z$  there exists a smooth sheet of small-amplitude nontrivial solutions 
\[
\SS^{\text{mixed}}_\delta := \{(\overline w(r,\upsilon),\overline \mu(r,\upsilon), \overline \alpha(r,\upsilon))\colon 0<r<\ve, \delta < |\upsilon| < \pi-\delta \}
\]
of \eqref{eq:problem two-dimensional} in $\OO \times \R^2$, passing 
through $(\overline w(0,\upsilon), \overline \mu(0,\upsilon), \overline \alpha(0,\upsilon))=(0, \mu^*, \alpha^*)$ with 
\[
\overline w(r,\upsilon)=r\cos(\upsilon) w_1^*+r\sin(\upsilon) w_2^* +O(r^2).
\]
In a neighborhood of $(0, \mu^*, \alpha^*)$ in $\OO \times \R^2$  the union $\SS^\text{mixed}_\delta \cup  \SS^{(k_2)}$ contains all nontrivial solutions of \eqref{eq:problem two-dimensional} such that $\delta < |\upsilon| < \pi-\delta$, where $r \cos(\upsilon)w_1^*+ r\sin(\upsilon)w_2^*$ is the projection of $w$ on $\ker \FF_w(0, \Lambda^*)$ parallel to $X_0$.
\end{itemize}
\end{theorem} 

\begin{rem}\label{rem:two-dimensional}
When $t_1\ne 0$ and $t_2\ne 0$ we find solutions which are neither $2\pi/k_1$ nor $2\pi/k_2$-periodic.
A more precise result could be obtained by studying the Taylor polynomial of $\Phi$ at $(0,\Lambda^*)$ of sufficiently high order as in e.g. \cite{JonesToland86, Jones89, JonesToland85, W09}. The computations are in general quite involved. We settle for the mere existence of multimodal solutions here.
\end{rem}

{\em Proof of Theorem \ref{thm:two-dimensional}}.
Define $\tilde w^*_j := ( \eta_{\phi_j^*},\phi_j^*) \in \tilde X$, $j=1,2$, and recall that $Z=\spn\{\tilde w^*_1, \tilde w^*_2\}$.
If $\Pi_{1} \Phi=\Phi_1 \tilde w_1^*$ and $\Pi_{2} \Phi=\Phi_2 \tilde w_2^*$, where $\Pi_{1}$ and $\Pi_{2}$ denote the projections onto $\spn \{\tilde w_1^*\}$ and $\spn \{\tilde w_2^*\}$, respectively, equation \eqref{eq:Lyapunov-Schmidt} takes the form
\begin{equation}
\label{eq:reduced system}
\begin{aligned}
\Phi_1(t_1, t_2, \Lambda)&=0,\\
\Phi_2(t_1, t_2, \Lambda)&=0.
\end{aligned}
\end{equation}
This is a system of two equations with five unknowns, and we note that it has the trivial solution $(0,0,\Lambda)$ for all $\Lambda\in \UU'$. 
\medskip

Assume first that $k_2/k_1 \not \in \Z$ and suppose that $t_1=0$ and $w^*=t_2 w_2^*$, $t_2\in \R$. 
An application of the Lyapunov-Schmidt reduction in the subspace of $2\pi/k_2$-periodic functions yields that $\psi(w^*, \Lambda)$ is then $2\pi/k_2$-periodic. Hence, 
\begin{equation}
\label{eq:Phi1 vanishing}
\Phi_1(0,t_2, \Lambda)=0, \qquad \text{ for all } t_2, \Lambda.
\end{equation}
Similarly,
\begin{equation}
\label{eq:Phi2 vanishing}
\Phi_2(t_1,0, \Lambda)=0, \qquad \text{ for all } t_1, \Lambda.
\end{equation}
Let $\Psi:=(\Psi_1, \Psi_2)$,
\[
\Psi_1(t_1, t_2, \Lambda):=\int_0^1 (\pa_{t_1} \Phi_1)(zt_1, t_2, \Lambda)\, dz, \quad \Psi_2(t_1, t_2, \Lambda):=\int_0^1 (\pa_{t_2} \Phi_2)(t_1, z t_2, \Lambda)\, dz.
\]
From \eqref{eq:Phi1 vanishing} and \eqref{eq:Phi2 vanishing} it follows that 
\eqref{eq:reduced system} is equivalent to that 
\begin{equation}
\label{eq:reduced system 2}
\begin{aligned}
t_1 \Psi_1(t_1, t_2, \Lambda)&=0,\\
t_2 \Psi_2(t_1, t_2, \Lambda)&=0.
\end{aligned}
\end{equation}
There are four possibilities. The trivial case $t_1, t_ 2 = 0$ corresponds to trivial solutions. When $t_1 = 0$ but $t_2 \neq 0$ the system reduces to $\Psi_2(0,t_2,\Lambda)=0$, the solutions of which are locally given by the set $\SS^{(k_2)}$; when $t_1 \neq 0$ but $t_2 = 0$ we similarly obtain the set $\SS^{(k_1)}$. It remains to investigate the solutions of $\Psi_1(t_1, t_2, \Lambda)=\Psi_2(t_1, t_2, \Lambda) = 0$ in a neighborhood of $(0,0,\Lambda^*)$. Clearly, when those exist, they will intersect and connect the sets $\SS^{(k_1)}$ and $\SS^{(k_2)}$.
Since 
\[
\pa_{t_1} \Phi(0,0,\Lambda^*)=\Pi_Z \Diff_w \FF(0, \Lambda^*) w_1^*=0,
\]
we have that 
\[
\Psi_1(0, 0, \Lambda^*)= \pa_{t_1} \Phi_1(0, 0, \Lambda^*)=0,
\]
Similarly,
\[
\Psi_2(0, 0, \Lambda^*)= \pa_{t_2} \Phi_2(0, 0, \Lambda^*)=0.
\]
We now apply the implicit function theorem to $\Psi$ at the point $(0,0, \Lambda^*)$. It suffices to prove that the matrix 
\[
\begin{pmatrix}
\pa_\mu \Psi_1(0,0,\Lambda^*) & \pa_\mu \Psi_2(0,0,\Lambda^*)\\
\pa_\alpha \Psi_1(0,0,\Lambda^*) &  \pa_\alpha \Psi_2(0,0,\Lambda^*)
\end{pmatrix}
\]
is invertible.
To this effect, we note that 
\[
\pa_\mu \Psi_1(0,0,\Lambda^*) =\pa_{t_1}\pa_\mu \Phi_1(0, 0, \Lambda^*),
\]
and
\[
\pa_{t_1}\pa_\mu \Phi(0,0,\Lambda^*)=\Pi_Z \Diff^2_{w\mu} \FF(0, \Lambda^*) w_1^*.
\]
We have that
\begin{align*}
\Pi_Z \Diff^2_{w\mu} \FF(0, \Lambda^*) w_1^*&=
 \frac{\langle \Diff^2_{w\mu} \FF(0,\Lambda^*) w_1^*, \tilde w_1^*\rangle_Y}{\| \tilde w_1^*\|_Y^2}  \tilde w_1^*,
\end{align*}
since $\Diff^2_{w\mu} \FF(0,\Lambda^*) w_1^*$ is proportional to $\cos(k_1 x)$ and therefore orthogonal to $\tilde w_2^*$. Consequently, 
\[
\pa_\mu \Psi_1(0,0,\Lambda^*) =\pa_{t_1}\pa_\mu \Phi_1(0, 0, \Lambda^*)=\frac{\langle \Diff^2_{w\mu} \FF(0,\Lambda^*) w_1^*,\tilde w_1^*\rangle_Y}{\| \tilde w_1^*\|_Y^2} .
\]
Using the observation \eqref{eq:derivativesum} we find that
\[
\pa_\mu \Psi_1(0,0,\Lambda^*)=\frac{\langle \Diff_\mu \LL(\Lambda^*) \phi_1^*, \tilde w_1^*\rangle_Y}{\| \tilde w_1^*\|_Y^2} 
\]
Similar arguments show that
\[
\pa_\mu \Psi_2(0,0,\Lambda^*)=\frac{\langle \Diff_\mu \LL(\Lambda^*) \phi_2^*, \tilde w_2^*\rangle_Y}{\| \tilde w_2^*\|_Y^2}, \quad 
\]
and
\[
\pa_\alpha \Psi_j(0,0,\Lambda^*)=\frac{\langle \Diff_\alpha \LL(\Lambda^*) \phi_j^*, \tilde w_j^* \rangle_Y}{\| \tilde w_j^*\|_Y^2}, \qquad j =1,2. 
\]
Thus,
\begin{align*}
\det \begin{pmatrix}
\pa_\mu \Psi_1(0,0,\Lambda^*) & \pa_\mu \Psi_2(0,0,\Lambda^*)\\
\pa_\alpha \Psi_1(0,0,\Lambda^*) &  \pa_\alpha \Psi_2(0,0,\Lambda^*)
\end{pmatrix}
=
C \det 
\begin{pmatrix}
\langle \Diff_{\mu} \LL(\Lambda^*) \phi^*_1 ,\tilde w^*_1 \rangle_Y & 
\langle \Diff_{\mu} \LL(\Lambda^*) \phi^*_2 ,\tilde w^*_2  \rangle_Y\\
\langle \Diff_{\alpha} \LL(\Lambda^*) \phi^*_1,\tilde w^*_1 \rangle_Y & 
\langle \Diff_{\alpha} \LL (\Lambda^*) \phi^*_2,\tilde w^*_2 \rangle_Y 
\end{pmatrix},
\end{align*}
where
$C=\| \tilde w_1^*\|_Y^{-2} \| \tilde w_2^*\|_Y^{-2}\ne 0$.
From \eqref{eq:partialLmu} we see that
 \begin{align*}
&\left\langle \Diff_{\mu} \LL (\Lambda^*) \phi^*_j ,\tilde w^*_j \right\rangle_Y\\
&=
\frac{\pi}{(\psi_0)_s(1)} \left[\theta_0\sin(\lambda)\cos^*(\theta_{k_j})+\left(-\theta_0^2\cos(\lambda)+\frac{1}{(\mu^*)^2\theta_0 \sin(\lambda)}\right)\frac{\sin^*(\theta_{k_j})}{\theta_{k_j}}\right]\frac{\sin^*(\theta_{k_j})}{\theta_{k_j}}\\
&= \underbrace{\frac{\pi}{(\psi_0)_s(1)}\left[\theta_0\sin(\lambda)a-\theta_0^2\cos(\lambda)
+\frac{1}{(\mu^*)^2\theta_0 \sin(\lambda)}\right]}_{=A} \left(\frac{\sin^*(\theta_{k_j})}{\theta_{k_j}}\right)^2\\
&= A \left(\frac{\sin^*(\theta_{k_j})}{\theta_{k_j}}\right)^2.
 \end{align*} 
A straightforward calculation shows that
 \begin{align*}
\Diff_{\alpha} \LL (\Lambda^*) \phi&=\left(\frac{\mu \sin(\lambda)}{2\theta_0}\phi_s-\left(\mu\cos(\lambda) - \frac{1}{2\mu\sin(\lambda) \theta_0^3}\right)\phi,-\phi\right).
\end{align*}
From this it follows that
\begin{align*}
\left\langle \Diff_{\alpha} \LL(\Lambda^*) \phi^*_j ,\tilde w^*_j \right\rangle_Y=B\left(\frac{\sin^*(\theta_{k_j})}{\theta_{k_j}}\right)^2+f(k_j),
\end{align*}
where $B=\frac{-\pi}{(\psi_0)_s(1)} \left[\frac{\mu \sin \lambda}{2\theta_0}a-\left(\mu\cos \lambda - \frac{1}{2\mu\sin \lambda \theta_0^3}\right)\right]$
and
\[
f(k_j):=\begin{cases}
\frac\pi2\sgn(k_j^2+\alpha)\frac{\theta_{k_j}-\cos^*(\theta_{k_j})\sin^*(\theta_{k_j})}{\theta_{k_j}^3}, & 
\theta_{k_j} \ne 0,\\
-\frac{\pi}{3}, & \theta_{k_j}=0.
 \end{cases}
 \]
We thus have that
\begin{align*}
&
\det \begin{pmatrix}
\left\langle \Diff_{\mu} \LL(\Lambda^*) \phi^*_1 ,\tilde w^*_1 \right\rangle_Y & 
\left\langle \Diff_{\mu} \LL (\Lambda^*) \phi^*_2 ,\tilde w^*_2  \right\rangle_Y\\
\left\langle \Diff_{\alpha} \LL(\Lambda^*) \phi^*_1,\tilde w^*_1 \right\rangle_Y & 
\left\langle \Diff_{\alpha} \LL(\Lambda^*) \phi^*_2,\tilde w^*_2 \right\rangle_Y 
\end{pmatrix}\\
&=\det \begin{pmatrix}
A(\sin^*(\theta_{k_1}))^2/\theta_{k_1}^2 & 
A(\sin^*(\theta_{k_2}))^2/\theta_{k_2}^2\\
B(\sin^*(\theta_{k_1}))^2/\theta_{k_1}^2+f(k_1) & 
B(\sin^*(\theta_{k_2}))^2/\theta_{k_2}^2+f(k_2) 
\end{pmatrix}\\
&= A\det \begin{pmatrix}
(\sin^*(\theta_{k_1}))^2/\theta_{k_1}^2 & 
(\sin^*(\theta_{k_2}))^2/\theta_{k_2}^2\\
f(k_1) & 
f(k_2) 
\end{pmatrix}\\
&=A\left(\frac{(\sin^*(\theta_{k_1}))^2}{\theta_{k_1}^2}f(k_2)-
\frac{(\sin^*(\theta_{k_2}))^2}{\theta_{k_2}^2}f(k_1)\right).
\end{align*}
Note that if $a=0$, then $k_j^2+\alpha<0$ and $\cos(\theta_{k_j})=0$,  $j=1,2$. Hence, 
$f(k_j)=-\pi/(2\theta_{k_j}^2)$ in that case, and
\[
\frac{(\sin^*(\theta_{k_1}))^2}{\theta_{k_1}^2}f(k_2)-
\frac{(\sin^*(\theta_{k_2}))^2}{\theta_{k_2}^2}f(k_1)=-\frac{\pi}{2\theta_{k_1}^2\theta_{k_2}^2}
(\sin^2(\theta_{k_1})-\sin^2(\theta_{k_2})) =0.
\]
We therefore assume that $a\ne 0$ from now on.
We consider the three different possible cases. 
\medskip

\noindent {Case I}. Assume that $k_1^2+\alpha<0$ and $k_2^2+\alpha<0$.
In this case we find that
\begin{align*}
&\frac{(\sin^*(\theta_{k_1}))^2}{\theta_{k_1}^2} f(k_2) - \frac{(\sin^*(\theta_{k_2}))^2}{\theta_{k_2}^2} f(k_1)\\
&=\frac{\pi}{2} \frac{1}{ \theta_{k_1}^2 \theta_{k_2}^2}\, a^{-1} \left[\sin^2(\theta_{k_2}) \left((a-1)+\sin^2(\theta_{k_1})\right)-
\sin^2(\theta_{k_1}) \left((a-1)+\sin^2(\theta_{k_2})\right)\right]\\
&=\frac{\pi}{2} \frac{1}{ \theta_{k_1}^2 \theta_{k_2}^2}\, a^{-1}(a-1) \left(\sin^2(\theta_{k_2})-
\sin^2(\theta_{k_1}) \right).
\end{align*}
This is non-zero if $a \ne 1$, in view of the equality $\theta_{k_1}\cot(\theta_{k_1})=\theta_{k_2}\cot(\theta_{k_2})$.
 \medskip

\noindent {Case II}. Assume that $k_1^2+\alpha<0$ and $k_2^2+\alpha=0$.
In this case $a=1$ and we find that
\begin{align*}
\frac{(\sin^*(\theta_{k_1}))^2}{\theta_{k_1}^2}f(k_2)-\frac{(\sin^*(\theta_{k_2}))^2}{\theta_{k_2}^2}f(k_1)
&=-\frac{\pi}{3}\frac{\sin^2(\theta_{k_1})}{\theta_{k_1}^2}+\frac{\pi}{2\theta_{k_1}^2}\left(1-\frac{\cos(\theta_{k_1})\sin(\theta_{k_1})}{\theta_{k_1}}\right)\\
&=\frac{\pi}{6}\frac{\sin^2(\theta_{k_1})}{\theta_{k_1}^2},
\end{align*}
which is non-zero by the argument used in Lemma \ref{lemma:Bifurcation condition}.
 \medskip
 
\noindent {Case III.} Assume that
$k_1^2 + \alpha <0$ while $k_2^2+\alpha>0$.
We then find that 
\begin{align*}
&\frac{(\sin^*(\theta_{k_1}))^2}{\theta_{k_1}^2} f(k_2) - \frac{(\sin^*(\theta_{k_2}))^2}{\theta_{k_2}^2} f(k_1)\\
&=\frac{\pi}{2} \frac{1}{ \theta_{k_1}^2 \theta_{k_2}^2}\, a^{-1} \left[\sin^2(\theta_{k_1})\left( (a-1)-\sinh^2(\theta_{k_2}) \right)+
\sinh^2(\theta_{k_2}) \left((a-1)+\sin^2(\theta_{k_1})\right)\right]\\
&=\frac{\pi}{2} \frac{1}{ \theta_{k_1}^2 \theta_{k_2}^2}\, a^{-1}(a-1)\left(\sin^2(\theta_{k_1})+
\sinh^2(\theta_{k_2})\right) \ne 0,
\end{align*}
when $a \ne 1$.
This concludes the proof of the first part of the theorem.
\medskip

Assume next that $k_2/k_1\in \Z$. In this case \eqref{eq:Phi1 vanishing} remains true, whereas
\eqref{eq:Phi2 vanishing} may be false. We introduce $\Psi_1$ as before, but write it in polar coordinates:
\[
\Psi_1(r,\upsilon, \Lambda)=\int_0^1 \pa_{t_1}\Phi_1(z r\cos(\upsilon), r\sin(\upsilon), \Lambda)\, dz.
\]
We do not know that $\Phi_2$ vanishes when $t_2=0$, but we can still use the identity $\Phi_2(0,0,\Lambda)=0$.  Redefine $\Psi_2$ as
\begin{align*}
&\Psi_2(r,\upsilon, \Lambda)\\
&:=\int_0^1\Bigl\{ \pa_{t_1}\Phi_2 (z r\cos(\upsilon), z r\sin(\upsilon), \Lambda)\cos(\upsilon)+\pa_{t_2}\Phi_2 (z r\cos(\upsilon), z r\sin(\upsilon), \Lambda)\sin(\upsilon)\Bigr\}\, dz.
\end{align*}
We then have that $r \Psi_2(r,\upsilon,\Lambda)=\Phi_2(r\cos(\upsilon), r\sin(\upsilon), \Lambda)$, and equation~\eqref{eq:reduced system} is equivalent to $(\Psi_1(r, \upsilon, \Lambda), \Psi_2(r, \upsilon, \Lambda))=(0,0)$ 
whenever $t_1 \ne 0$. Moreover, $\Psi_1(0, \upsilon, \Lambda)=\Psi_2(0, \upsilon, \Lambda)=0$ for all $\Lambda$. The complete set of solutions is locally given by the three cases $r=0$, $\cos(\upsilon) = 0$ with $\Psi_2(r,\upsilon,\Lambda) = 0$, and $\Psi_1(r,\upsilon,\Lambda) = \Psi_2(r,\upsilon, \Lambda) = 0$; the first set is trivial, the second given by $\SS^{(k_2)}$, and the third yet to be found. By considering $\partial_r \Phi(0,v,\Lambda^*)$, as before we have that 
\[
\pa_\mu \Psi_1(0,\upsilon,\Lambda^*)=\frac{\langle \Diff_\mu \LL(\Lambda^*) \phi_1^*, \tilde w_1^*\rangle_Y}{\| \tilde w_1^*\|_Y^2}, 
\]
while
\begin{align*}
\pa_\mu \Psi_2(0,\upsilon,\Lambda^*)&=\frac{\langle \Diff_\mu \LL(\Lambda^*) \phi_1^*, \tilde w_2^*\rangle_Y}{\| \tilde w_2^*\|_Y^2}\cos(\upsilon)
+ \frac{\langle \Diff_\mu \LL(\Lambda^*) \phi_2^*, \tilde w_2^*\rangle_Y}{\| \tilde w_2^*\|_Y^2}\sin(\upsilon)\\
&=\frac{\langle \Diff_\mu \LL(\Lambda^*) \phi_2^*, \tilde w_2^*\rangle_Y}{\| \tilde w_2^*\|_Y^2}\sin(\upsilon),
\end{align*}
where we have used the fact that $\Diff_\mu \LL(\Lambda^*) \phi_1^*$ is orthogonal to $\tilde w_2^*$. Similarly (cf.~\ref{eq:Lexpression}),
\[
\pa_\alpha \Psi_1(0,\upsilon,\Lambda^*)=\frac{\langle \Diff_\alpha \LL(\Lambda^*) \phi_1^*, \tilde w_1^*\rangle_Y}{\| \tilde w_1^*\|_Y^2} 
\quad \text{ and } \quad
\pa_\alpha \Psi_2(0,\upsilon,\Lambda^*)=\frac{\langle \Diff_\alpha \LL(\Lambda^*) \phi_2^*, \tilde w_2^*\rangle_Y}{\| \tilde w_2^*\|_Y^2}\sin(\upsilon).
\]
Thus,
\begin{align*}
&\det \begin{pmatrix}
\pa_\mu \Psi_1(0,\upsilon,\Lambda^*) & \pa_\mu \Psi_2(0,\upsilon,\Lambda^*)\\
\pa_\alpha \Psi_1(0,\upsilon,\Lambda^*) &  \pa_\alpha \Psi_2(0,\upsilon,\Lambda^*)
\end{pmatrix}\\
&=
C\sin(\upsilon) \det 
\begin{pmatrix}
\langle \Diff_{\mu} \LL(\Lambda^*) \phi^*_1 ,\tilde w^*_1 \rangle_Y & 
\langle \Diff_{\mu} \LL(\Lambda^*) \phi^*_2 ,\tilde w^*_2  \rangle_Y\\
\langle \Diff_{\alpha} \LL(\Lambda^*) \phi^*_1,\tilde w^*_1 \rangle_Y & 
\langle \Diff_{\alpha} \LL (\Lambda^*) \phi^*_2,\tilde w^*_2 \rangle_Y 
\end{pmatrix},
\end{align*}
where
$C=\| \tilde w_1^*\|_Y^{-2} \| \tilde w_2^*\|_Y^{-2}\ne 0$. We can therefore apply the implicit function theorem under the assumptions of Theorem~\ref{thm:two-dimensional} as long as $\sin(\upsilon)\ne 0$. This gives us local, smoothly parametrized, solution sets at every point $(r,\upsilon) = (0,\upsilon_0)$ with $0 < \upsilon_0 < \pi$. On each compact interval in $[\delta, \pi-\delta] \subset (0,\pi)$ we may find a positive number $\varepsilon$ such the parametrization is uniformly valid for $0 < r < \varepsilon$, $\delta < \upsilon < \pi-\delta$.
\endproof

\begin{rem}
By varying $\lambda$ in a neighborhood of $\lambda^*$ we obtain a smooth family of bifurcating two-dimensional sheets of non-trivial solutions.
This is quite natural if one considers the bifurcation equation 
\eqref{eq:bifurcation2}: if $\lambda$ is varied near $\lambda^*$ one can adjust $\mu$ to make sure that the kernel remains two-dimensional.
One can check that the theorem also holds if $(\lambda, \alpha)$ are used as bifurcation parameters instead of $(\mu, \alpha)$. However, the vectors
$(\pa_\mu \Psi_1(0,0,\Lambda^*),\pa_\mu \Psi_2(0,0,\Lambda^*))$ and $(\pa_\lambda \Psi_1(0,0,\Lambda^*), \pa_\lambda \Psi_2(0,0,\Lambda^*))$ are parallel, and it is therefore not 
possible to replace $(\mu, \alpha)$ with $(\mu, \lambda)$ (the same problem appears when $k_2/k_1\in \Z$). Again this can be explained by looking at 
\eqref{eq:bifurcation2}. If the result were true with $(\mu, \lambda)$ as bifurcation parameters, we would obtain a family of bifurcating two-dimensional sheets indexed by $\alpha$. 
However,  Lemma \ref{lemma:kernels} shows that if $\alpha$ is varied near $\alpha^*$, it is impossible to adjust $\mu$ and $\lambda$ to keep the kernel two-dimensional.
\end{rem}

\section{Applications and examples}\label{sec:applications}
The waves found in this paper are all small-amplitude rotational gravity waves in water of finite depth. For simplicity we shall take $\kappa=1$ here so that the basic period is $2\pi$.

\begin{corollary}[Stokes waves with arbitrarily many critical layers]
\label{cor:one-dimensional}
Let $\beta \in (0,1)$, $k \in \Z^+$. For almost every $\alpha < -\pi^2$ there exists a $C^\infty$-curve 
\[
t \mapsto \left[ \eta, \hat\psi, Q, m_0, m_1 \right](t) \in C^{2+\beta}(\s) \times C^{2+\beta}(\overline{\hat\Omega}) \times \R^3 
\]
corresponding to even, $2\pi/k$-periodic solutions $(\eta,\psi,Q,m_0,m_1)$ of the steady water-wave problem \eqref{eq:problem}. 
These have a surface profile which rises and falls exactly once in every minimal period, and at least one critical layer. The number of critical layers can be chosen arbitrarily large by choosing $-\alpha$ large enough.
\end{corollary}

\begin{proof}
Lemma~\ref{lemma:kernels}~i) yields the existence of $\mu^*$, $\lambda^*$ and, via Lemma~\ref{lemma:Lproperties}, the one-dimensional kernel 
$w^* = \TT(\Lambda^*) \cos(kx) \sin^*(\theta_k s)/\theta_k$ needed for Theorem~\ref{thm:one-dimensional} (here $\alpha^* := \alpha$). By identifying $m_1 := \mu(t) \cos(\lambda)$ and 
$m_2 := \mu(t) \cos(\lambda - \sqrt{|\alpha^*|})$, we obtain from Lemma~\ref{lemma:equivalence} the solutions of the original problem, as well as the $C^\infty$-dependence with respect to 
$t$. Notice that also $Q = Q(\mu(t),\alpha^*,\lambda^*)$ depends on the bifurcation parameter $t$ through the formula \eqref{eq:Q}. 
If we identify $w(t) = \TT(t) \phi(t)$ we have $\Diff_t \eta(0) = \eta_{\left(\Diff_t \phi(0)\right)}$:
\[
\Diff_t \left( \frac{-\phi(t)}{(\psi_0)_s(t)}\right)\bigg|_{t = 0} = - \left( \frac{\Diff_t \phi(t)}{(\psi_0)_s(t)} + \phi(t) \Diff_t \frac{1}{(\psi_0)_s(t)} \right)\bigg|_{t = 0} = - \frac{\Diff_t \phi(0)}{(\psi_0)_s(0)},
\]
since $t = 0$ corresponds to the trivial solution $\phi(0) = 0$. In view of that
\[
\eta_{\Diff_t \phi(0)} = \frac{\sin^*(\theta_k) \cos(kx)}{\theta_k \mu^* \theta_0 \sin(\lambda)} =: C(\Lambda^*, k) \cos(kx),
\]
we have $C(\Lambda^*, k) \neq 0$ and  the local expression
\[
\eta(x;t) = t\, C(\Lambda^*, k) \cos(kx) + O \left( t^2 \right), \qquad\text{ as }\quad t \to 0, 
\]
in the $C^{2+\beta}(\s,\R)$-metric. This guarantees that the surface profile has one crest and one trough per period for $t$ small enough. The number of critical layers for small enough solutions along the bifurcation curve can be seen directly from the number of zeros of $(\psi_0)_y = -\mu \theta_0 \sin(\theta_0(y-1) + \lambda)$ (for details, see~\cite{eev10}). 
\end{proof}

\begin{figure}
\begin{center}
\includegraphics[clip=true,trim = 80 0 65 25, width=0.4\linewidth]{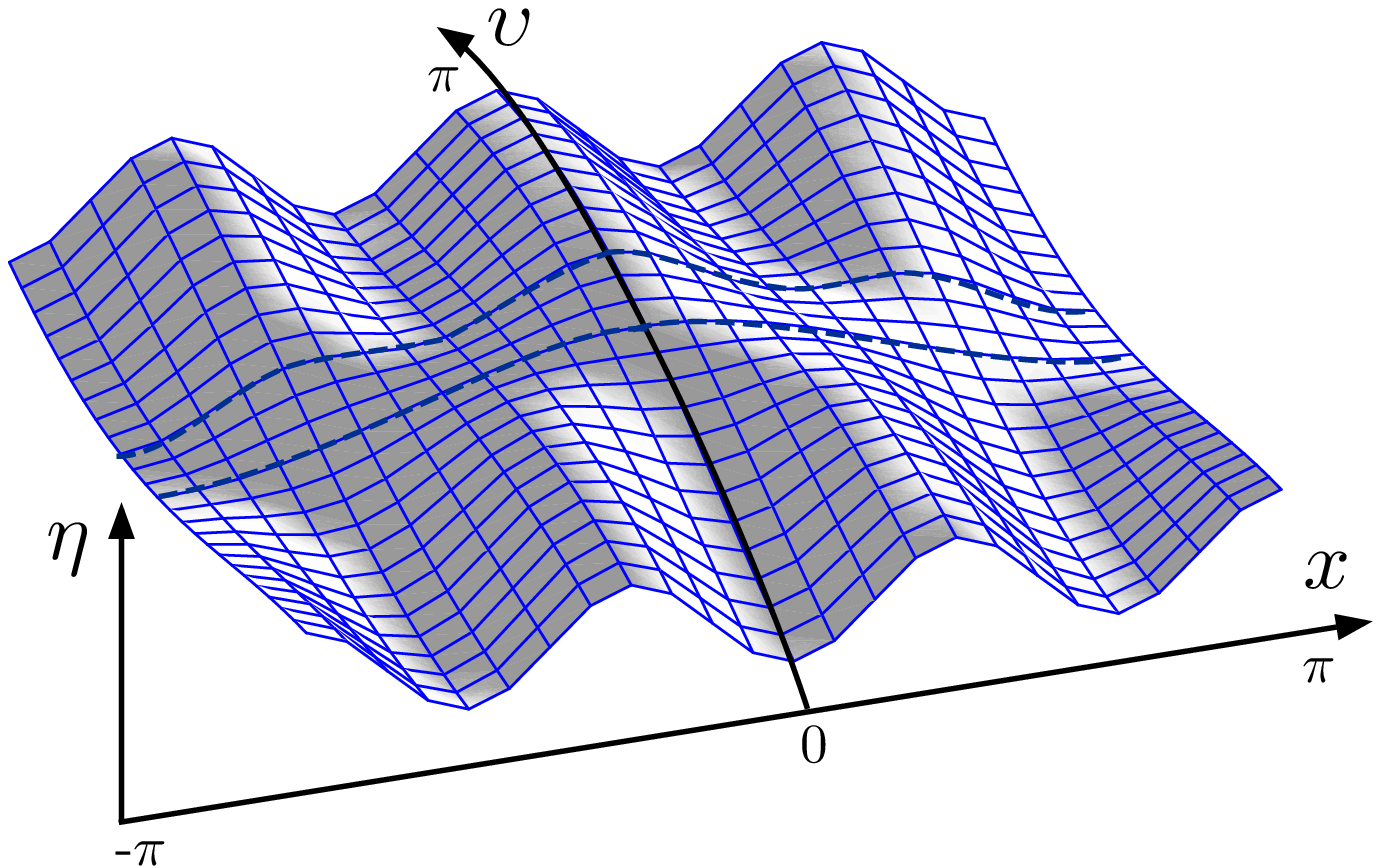}
\quad
\includegraphics[width=0.4\linewidth]{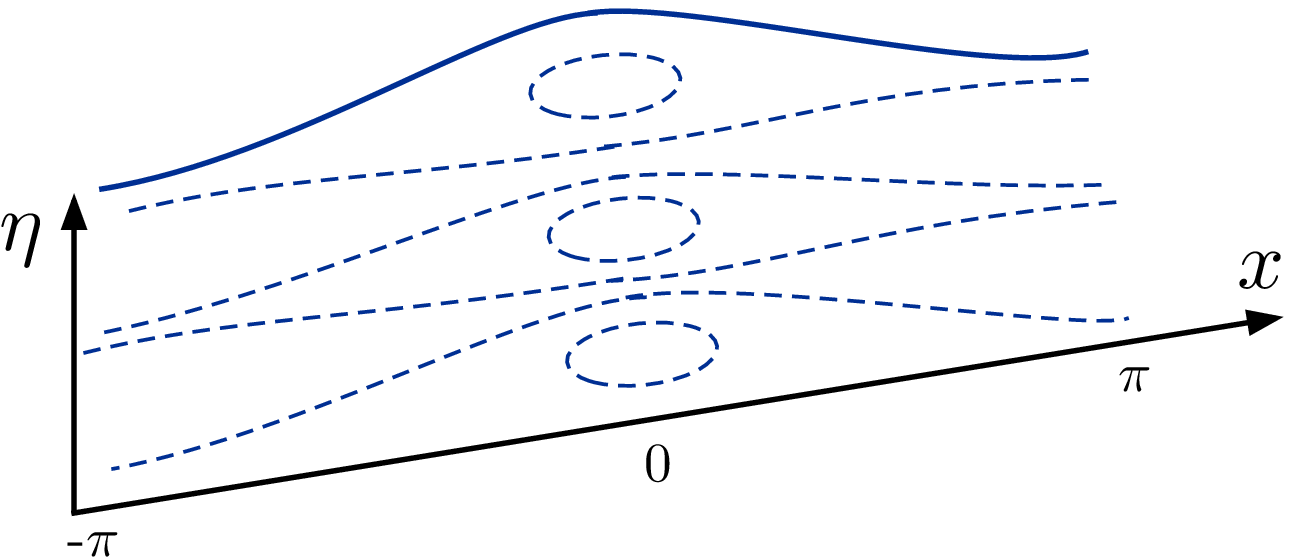}
\end{center}
\caption{\small{\bfseries Rotational gravity water waves bifurcating from one- and two-dimensional kernels.} Left: the qualitative features of surface profiles for the small-amplitude waves found in Corollary~\ref{cor:sin-sinh} along the bifurcation parameter $\upsilon$. The marked lines show: at $\upsilon = \pi/2$ a typical Stokes wave with a single crest within each  period; at $\upsilon \approx 2\pi/3$ a wave with several crests within each minimal period. Right: one typical rotational behavior of the waves found in Corollary~\ref{cor:one-dimensional} bifurcating from a single eigenvalue. The dotted lines indicate streamlines. Notice that the number of vortices can be chosen arbitrarily large.}  
\label{fig:two-dim}
\end{figure}

\begin{corollary}[Doubly-periodic waves with at least one critical layer]
\label{cor:sin-sinh}
Let $\beta \in (0,1)$ and $k \geq 5$ be an integer. There exists a family
\begin{equation*}\label{eq:doublewaves}
(r,\upsilon) \mapsto \left[ \eta,\psi,Q,m_0,m_1,\alpha \right](r,\upsilon) \in C^{2+\beta}(\s) \times C^{2+\beta}(\overline\Omega) \times \R^3 \times (-k^2,-\pi^2-1), 
\end{equation*}
of nontrivial, even, $2\pi$-periodic solutions of the steady water-wave problem \eqref{eq:problem} with at least one critical layer. For each $\upsilon \in (0,\pi)$ their surface elevation satisfies
\begin{equation}\label{eq:twodimsurface}
\eta(x;r) =  r C_1 \cos(\upsilon) \cos(x) +  r C_2 \sin(\upsilon) \cos(kx) + O\left( r^2 \right), \quad C_1, C_2 \neq 0,
\end{equation}
in the $C^{2+\beta}(\s,\R)$-metric as $r \to 0$. 
\end{corollary}

\begin{proof}
The assumption on $k$ guarantees that we may pick any $\lambda^*$ with $\cot(\lambda^*)\le 0$ and use Lemma~\ref{lemma:kernels}~ii) with $k_1 = 1$, $k_2 = k$, 
yielding $\mu^*$ and $\alpha^*$. The proof of Lemma~\ref{lemma:kernels}~ii) is devised so as to have $a \not\in \{0,1\}$ (this is \eqref{eq:aneq1}). 
That, in turn, allows us to apply Theorem~\ref{thm:two-dimensional} ii). The constants $m_0$ and $m_1$ are identified as in the proof of Corollary~\ref{cor:one-dimensional} 
where now also $\overline \alpha(r,\upsilon)$ depends on the bifurcation parameters, and \eqref{eq:Q} means that $Q = Q(\overline \mu(r,\upsilon),\overline \alpha(r,\upsilon), \lambda)$. 
Since $\alpha^* \in (-k^2,-\pi^2-1)$ according to Lemma~\ref{lemma:kernels}, and $\overline \alpha(0,\upsilon) = \alpha^*$ for all $\upsilon \in (0,\pi)$, 
we still have $\overline \alpha(r,\upsilon) \in (-k^2,-\pi^2-1)$ if, for each $\upsilon \in (0,\pi)$, we choose $|r|$ small enough. 
The formula \eqref{eq:twodimsurface} is obtained from $\partial_r \overline w(0,\upsilon)=\cos(\upsilon) \TT(\Lambda^*)\phi_1^*+\sin(\upsilon) \TT(\Lambda^*)\phi_2^*$ just as in 
the proof of Corollary~\ref{cor:one-dimensional}, and the existence of a critical layer follows from \eqref{eq:trivial} and the fact that $\alpha < -\pi^2$. 
\end{proof}

\begin{example}[Doubly-periodic waves with several critical layers]
\label{ex:sin-sinh}
There exists a family of nontrivial, even, $2\pi$-periodic solutions of the steady water-wave problem \eqref{eq:problem}, with approximate surface elevation 
\[
\eta(x) = C_1 t_1 \cos(4x) +  C_2 t_2 \cos(7x), \quad C_1, C_2 \neq 0.
\]
and three critical layers. 
\end{example}

{\em Sketch of proof}.
One checks that $k_1 = 4$, $k_2 = 7$, and $\alpha^* \approx -60$ fulfills \eqref{eq:a} for $a \approx 18$. 
Those $k_i$'s are thus an example when Lemma~\ref{lemma:kernels}~iii) holds exactly for $k_1$ and $k_2$. 
The proof then follows the lines of that of Corollary~\ref{cor:sin-sinh}, except that we now use Theorem~\ref{thm:two-dimensional} i). 
Since $\sqrt{|\alpha^*|} > 2\pi$ it follows from \eqref{eq:trivial} that, for $t_1^2+t_2^2$ small enough, the resulting waves have at least two critical layers. By choosing first $\lambda$ and then $\mu$ appropriately in \eqref{eq:bifurcation2} one obtains three critical layers.
\endproof

\section*{Acknowledgments}

Part of this research was conceived during the workshop on \emph{Wave Motion} in Oberwolfach, spring $2009$; M. E. and E. W. thank the organizers for their kind invitation. 
E. W. was supported by the Alexander von Humboldt foundation. The suggestions of the referees are gratefully acknowledged.

\bibliographystyle{siam}
\bibliography{critical}

\end{document}